\documentclass[11pt,leqno,oneside]{amsart}

\usepackage[utf8]{inputenc} %
\usepackage[T1]{fontenc}    %
\usepackage{microtype} 
\usepackage[english]{babel} %

\usepackage[a4paper,margin=1in,marginpar=1in]{geometry} %
\usepackage{csquotes}       %

\usepackage[shortlabels,inline]{enumitem}

\usepackage{amsmath}   %
\usepackage{amsthm}    %
\usepackage{mathtools} %
\usepackage[lcgreekalpha]{stix2}          %

\usepackage[
backend=biber,        %
style=ext-alphabetic,     %
giveninits=true,      %
maxalphanames=5,      %
minalphanames=5,      %
maxbibnames=99,       %
articlein=false,
]{biblatex}

\AtEveryBibitem{%
  \clearfield{note}%
  \clearlist{location}%
  \clearlist{language}%
}

\DeclareSourcemap{
  \maps[datatype=bibtex]{
    \map{
      \step[fieldsource=doi,final]
      \step[fieldset=url,null]
    }  
  }
}

\usepackage{xcolor}                   %
\definecolor{green}{HTML}{2ECC71}
\definecolor{blue}{HTML}{3498DB}
\definecolor{red}{HTML}{E74C3C}
\definecolor{orange}{HTML}{FD6A02}
\usepackage[pdfpagelabels]{hyperref}  %
\hypersetup{%
  colorlinks,
  linktocpage,
  urlcolor  = blue,
  citecolor = green,
linkcolor = red}

\usepackage[capitalise,sort]{cleveref} %

\AfterEndEnvironment{proof}{\noindent\ignorespaces} %
\makeatletter
\def\@endtheorem{\endtrivlist}
\makeatother

\Crefname{paragraph}{\S}{\SS}
\Crefname{equation}{}{}
\Crefname{enumi}{}{}
\Crefname{conditioni}{Condition}{Conditions}
\Crefname{conditionalti}{Condition}{Conditions}
\newtheorem{theorem}{Theorem}[section]
\newtheorem*{theorem*}{Theorem}
\Crefname{theorem}{Theorem}{Theorems}

\Crefname{theoremintro}{Theorem}{Theorems}

\newtheorem{lemma}[theorem]{Lemma}
\Crefname{lemma}{Lemma}{Lemmas}
\newtheorem{proposition}[theorem]{Proposition}
\Crefname{proposition}{Proposition}{Propositions}

\Crefname{corollary}{Corollary}{Corollaries}

\Crefname{conjecture}{Conjecture}{Conjectures}

\theoremstyle{definition}

\Crefname{example}{Example}{Examples}
\newtheorem*{example*}{Example}
\Crefname{assumption}{Assumption}{Assumptions}

\Crefname{definition}{Definition}{Definitions}

\Crefname{question}{Question}{Questions}
\theoremstyle{remark}
\newtheorem{remark}[theorem]{Remark}
\Crefname{remark}{Remark}{Remarks}

\numberwithin{equation}{section} %

\usepackage{booktabs} %

\DeclarePairedDelimiter{\paren}{\lparen}{\rparen}
\DeclarePairedDelimiter{\bracket}{\lbrack}{\rbrack}
\DeclarePairedDelimiter{\set}{\lbrace}{\rbrace}
\DeclarePairedDelimiter{\abs}{\lvert}{\rvert}

\DeclarePairedDelimiter{\norm}{\lVert}{\rVert}
\DeclarePairedDelimiterX{\psh}[2]{\langle}{\rangle}{#1, #2}
\DeclarePairedDelimiterX{\pairing}[2]{\langle}{\rangle}{#1 \vert #2}

\DeclarePairedDelimiterXPP{\Exp}[1]{\exp}{\lparen}{\rparen}{}{#1}
\DeclarePairedDelimiterXPP{\Log}[1]{\log}{\lparen}{\rparen}{}{#1}
\DeclarePairedDelimiterXPP{\Inf}[1]{\inf}{\lbrace}{\rbrace}{}{#1}
\DeclarePairedDelimiterXPP{\Sup}[1]{\sup}{\lbrace}{\rbrace}{}{#1}
\DeclarePairedDelimiterXPP{\Max}[1]{\max}{\lbrace}{\rbrace}{}{#1}
\DeclarePairedDelimiterXPP{\Min}[1]{\min}{\lbrace}{\rbrace}{}{#1}

\DeclareMathOperator{\esp}{\mathbf{E}}
\DeclareMathOperator{\prob}{\mathbf{P}}
\DeclareMathOperator{\var}{\mathbf{Var}}
\DeclareMathOperator{\law}{\mathbf{law}}

\newcommand{\given}[1][]{%
  \nonscript\:#1\vert
  \allowbreak
  \nonscript\:
\mathopen{}}

\DeclarePairedDelimiterXPP{\Prob}[1]{\prob}[]{}{#1}
\DeclarePairedDelimiterXPP{\Esp}[1]{\esp}[]{}{#1}
\DeclarePairedDelimiterXPP{\Var}[1]{\var}[]{}{#1}
\DeclarePairedDelimiterXPP{\Law}[1]{\law}[]{}{#1}

\author[R.\ Herry]{Ronan HERRY}
\address{IRMAR, Université de Rennes 1}
\email{ronan.herry@univ-rennes.fr}
\urladdr{https://orcid.org/0000-0001-6313-1372}
\author[B.\ Huguet]{Baptiste HUGUET}
\address{IRMAR, Université de Rennes 1}
\email{Baptiste.Huguet@math.cnrs.fr}
\urladdr{https://orcid.org/0000-0003-3211-3387}

\title[Brenier--Schrödinger with respect to Feller semimartingales]{The Brenier--Schrödinger problem with respect to Feller semimartingales and non-local Hamilton--Jacobi--Bellman equations}

\begin{document}

\maketitle
\begin{abstract}
  Motivated by a problem from incompressible fluid mechanics of Brenier \cite{Brenier}, and its recent entropic relaxation by \cite{ACLZ}, we study a problem of entropic minimization on the path space when the reference measure is a generic Feller semimartingale.
  We show that, under some regularity condition, our problem connects naturally with a, possibly non-local, version of the Hamilton--Jacobi--Bellman equation.
  Additionally, we study existence of minimizers when the reference measure in a Ornstein--Uhlenbeck process.
\end{abstract}

\setcounter{tocdepth}{1}
\tableofcontents

\section{Introduction}

\subsection{Main result}
We study the so-called \emph{Brenier--Schrödinger problem} with respect to a reference measure $\mathsf{R}$, that is the law of a Feller martingale on $[0,1]$.
This problem consists in minimizing the relative entropy of the law $\mathsf{P}$ of another process $Y$ under the marginal constraints:
\begin{enumerate}[(i)]
  \item $Y_{t} \sim \mu_{t}$, for all $t \in [0,1]$, where $(\mu_{t})$ is the data of a family of probability measure.
  \item $(Y_{0}, Y_{1}) \sim \pi$, where $\pi$ is the data of a coupling of the endpoints.
\end{enumerate}
By definition of the relative entropy, every minimizer is absolutely continuous with respect to $\mathsf{R}$.
It particular, $\mathsf{P}$ is also the law of a semimartingale, and the \emph{log-density process $Z$} is a semimartingale.
Actually, conditionally on $\{X_{0} = x\}$, $Z$ is of the form
\begin{equation*}
  Z_{t} = A_{t^{-}} + \psi_{t}^{x}(X_{t}),
\end{equation*}
where $A$ is an additive functional, and $(t, z) \mapsto \psi_{t}^{x}(z)$ is some function.
We say that the minimizer $\mathsf{P}$ is \emph{regular} provided the associated $\psi^{x}$ is $\mathscr{C}^{1}$ in time, $\mathscr{C}^{2}$ in space, and the semimartingale $\psi^{x}_{\cdot}(X)$ is \emph{special} (see definitions below).
We recall that a semimartingale is special as soon as it has bounded jumps.
In particular, every continuous semimartingale is special.
Our main result is as follows.
\begin{theorem*}
  Consider a regular solution $\mathsf{P}$ of the aforementioned problem, and the associated $\psi^{x}$.
  Then, the additive functional $A$ is absolutely continuous.
  In particular, there exists a function $p \colon [0,1] \times \mathbb{R}^{n} \to \mathbb{R}$ such that
  \begin{equation}\label{eq:existence-pressure}
    A_{t} = \int_{0}^{t} p_{t}(X_{s}) \mathtt{d} s, \qquad t \in [0,1].
  \end{equation}
Moreover, $\psi^{x}$ is a solution to the \emph{generalized Hamilton--Jacobi--Bellman} equation
\begin{equation}\label{eq:hjb}
  \partial_{t} \psi^{x} + \mathrm{e}^{-\psi^x_t}\mathbf{A}\left(\mathrm{e}^{\psi^x_t}\right) + p_{s}=0, 
\end{equation}
where $\mathbf{A}$ denotes the Markov generator of $\mathsf{R}$.
\end{theorem*}

\begin{remark}
  Let us make some comments on the result:
  \begin{enumerate}[(a),wide,nosep]
    \item 
      When $\mathsf{R}$ is a Markov diffusion, then $\mathbf{A}u(x) = \frac{1}{2} c(x) \cdot \nabla^{2} u(x) + b \cdot \nabla u(x)$, for some diffusion matrix $c$ and drift vector $b$.
      By the chain rule \cref{eq:hjb} becomes
      \begin{equation*}
        \partial_{t} \psi^{x}_{t} + \mathbf{A} \psi^{x} + \frac{1}{2} \nabla \psi^{x}_{t} \cdot c \nabla \psi^{x}_{t} + p_{t} = 0,
      \end{equation*}
      which is the usual Hamilton--Jacobi--Bellman equation with pressure.
    \item The function $p$ is interpreted as a pressure field.
    \item In the classical Hamilton--Jacobi--Bellman equation, the term $\mathbf{A}$ would have a negative sign. This phenomenon also has been observed in \cite{ACLZ}. It can be circumvent by considering the potential $\psi$ associated to the time-reversed measure.
  \end{enumerate}
\end{remark}

Verifying that there exist regular minimizers is an arduous task, that, in the Brownian case amounts to solve the Navier--Stokes equation (see \cite{ACLZ} for details).
We do not take on such an accomplishment.
More modestly, additionally to our main theorem, we show that there exists a minimizer, possibly non-regular, when $\mathsf{R}$ is the law of an Ornstein--Uhlenbeck process.

\subsection{Motivations and connections to the literature}
Understanding what happens to the Brenier--Schrödinger problem for general semimartingales, possibly with a jump part is the main motivation for this paper.
The Brenier--Schrödinger problem, defined in \cite{ACLZ}, is a relaxation of Brenier's approach \cite{Brenier} to incompressible perfect fluids and Euler equations.
This generalization, which can be seen as the entropic relaxation of Brenier original problem, aims at modelling viscous fluid  dynamics.
The achievements of \cite{ACLZ} are threefold.
\begin{enumerate}[(i)]
  \item When the reference measure $\mathsf{R}$ is \emph{Markovian}, they study the general shape of minimizers.
  \item  Whenever the reference measure $\mathsf{R}$ is the law of a reversible Brownian motion on $\mathbb{R}^{n}$ or $(\mathbb{R} \setminus \mathbb{Z})^{n}$, they show that minimizers of a certain form give rise to solution to the Navier--Stokes equation.
  \item On the torus, they show that minimizers always exists, whenever $\mu_{t} = \mathsf{vol}$ for all $t$, and the coupling $\pi$ has finite entropy relatively to the $\mathsf{R}_{01}$.
\end{enumerate}
The two last points are extended to the compact manifold setting in \cite{GarciaZeladaHuguet}.
Let us stress that the particular form of the minimizer needed in \cite{ACLZ} is stronger than our regularity condition.
Indeed, they need to assume the existence of the function $p$ such that \cref{eq:existence-pressure} holds, while \cref{eq:existence-pressure} follows from our analysis.
Thus, even in the purely continuous case, our result is less restrictive  than that of \cite{ACLZ}.

In the Brownian case, when the noise tends to $0$, one recovers in the limit Brenier's original problem of minimization of the kinetic energy, as illustrated in \cite{BaradatMonsaingeon}.
In our possibly non-local setting, understanding the appropriate notion of small noise limit, and whether there exists a equivalent of Brenier's problem is an interesting question that could be explored in future works.

\subsection{Outline of the proof}
Our proof follows the lines of \cite{ACLZ}.
We use, on the one hand, Itō's formula, and on the other hand, Girsanov theorem to give two different representations of the semimartingale $Z$.
At the technical level, since $Z$ is a special semimartingale, it thus admits a unique decomposition as a sum of a local martingale, and predictable process of finite variations. 
This allows us to identify the two different decompositions and conclude.
In the continuous case, as for the Brownian setting of \cite{ACLZ}, the semimartingale is always special and this assumption is unnecessary.

\section{Reminders and notation}

\subsection{Semimartingales and their characteristics}

We refer to \cite{JacShi} for more details.
\subsubsection{Path space}

Let $\Omega$ denote the space of right-continuous with left-limit paths from $[0,1]$ to $\mathbb{R}^n$.
The \emph{canonical process} on $\Omega$ is denoted by $X$, that is
\begin{equation*}
  X_{t}(\omega) \coloneq \omega_{t}, \qquad \omega \in \Omega, t \in [0,1].
\end{equation*}
The associated \emph{canonical filtration} defined by
\begin{equation*}
  \mathfrak{F}_{t} \coloneq \bigcap_{s > t} \sigma(X_{r} : 0 \leq r \leq s), \qquad t \in [0,1].
\end{equation*}
Conveniently for $\mathsf{P} \in \mathscr{P}(\Omega)$, we write $\mathsf{P}_{t}$ for the law of $X_{t}$ under $\mathsf{P}$, $\mathsf{P}_{ts}$ for the law of $(X_{t}, X_{s})$ under $\mathsf{P}$, and so on.

\subsubsection{Processes}
An \emph{adapted process} is a mapping $Y \colon \Omega \to \Omega$ such that $Y_{t}$ is $\mathfrak{F}_{t}$-measurable for all $t \in [0,1]$.
By definition, all our processes are right-continuous with left-limit.
For such $Y$, its \emph{jump process} is
\begin{equation*}
  \Delta Y_{t} \coloneq Y_{t} - Y_{t-}, \qquad t \in [0,1].
\end{equation*}
We say that $Y$ is \emph{predictable} whenever seen as the mapping
\begin{equation*}
  [0,1] \times \Omega \ni (t,\omega) \mapsto Y_{t}(w),
\end{equation*}
it is measurable with respect to the $\sigma$-algebra generated by $F \times (t,s]$ for $s < t$ and $F \in \mathfrak{F}_{s}$.
A process of \emph{finite variation} is an adapted process $A = B - C$ with $B$ and $C$ two non decreasing adapted processes.
In that case the \emph{variation} of $A$ is the process $B + C$.
We also consider \emph{generalized adapted} (resp.\ \emph{predictable}) processes of the form $W \colon \Omega \times \mathbb{R}^{n} \to \Omega$.

\subsubsection{Martingales \& semimartingales}
Fix $\mathsf{P} \in \mathscr{P}(\Omega)$.
A \emph{martingale} (or a \emph{$\mathsf{P}$-martingale} to emphasize the dependence on $\mathsf{P}$) is an adapted process such that:
\begin{equation*}
  \esp_{\mathsf{P}} \bracket{ M_{t} \given \mathfrak{F}_{s} } = M_{s}, \quad 0 \leq s \leq t \leq 1.
\end{equation*}
We say that $X$ is a \emph{semimartingale} under $\mathsf{P}$ provided
\begin{equation}\label{eq:semimartingale}
  X = X_{0} + M + A,
\end{equation}
where $M$ is a $\mathsf{P}$-martingale, and $A$ is a finite variation process ($\mathsf{P}$-almost surely).
The decomposition \cref{eq:semimartingale} is not unique.

\subsubsection{Special semimartingales}
If $A$ can be chosen predictable in \cref{eq:semimartingale}, we say that the semimartingale is \emph{special}.
By \cite[I.4.23-24]{JacShi}, a semimartingale is special if and only if, in \cref{eq:semimartingale}, $A$ can be chosen with $\mathsf{P}$-integrable variation, and that a martingale is special as soon as its jump are bounded.

\subsubsection{Canonical jumps measure}

We consider the \emph{canonical jumps process} and the \emph{canonical jumps measure}
\begin{align*}
  & \Delta X_{s} \coloneq X_{s} - X_{s-}, \quad s \in [0,1];
\\& \mu_{t} \coloneq \sum_{s \leq t} \delta_{(s, \Delta X_{s})} \mathbf{1}_{\Delta X_{s} \neq 0}.
\end{align*}
The measure $\mu$ is a random element of $\mathscr{M}([0,1] \times \mathbb{R}^{d})$, the set of Borel measures on $[0,1] \times \mathbb{R}^{d}$.
For a, possibly random, measure $\lambda \in \mathscr{M}([0,1] \times \mathbb{R}^{n})$, and a generalized process $W$, we write
\begin{equation*}
  (W * \lambda)_{t} \coloneq \int_{0}^{t} \int_{\mathbb{R}^{n}} W_{s}(y) \lambda(\mathtt{d} s\mathtt{d}y), \qquad t \in [0,1].
\end{equation*}

\subsubsection{Compensator of the jumps measure}
For all $\mathsf{P} \in \mathscr{P}(\Omega)$, there exists a unique predictable measure $\nu$ such that
\begin{equation*}
  \bracket*{ \esp_{\mathsf{P}} \bracket*{ \abs{W} * \mu } < \infty } \Rightarrow \bracket*{ \esp_{\mathsf{P}} \bracket*{ \abs{W} * \nu } < \infty, \ \text{and}\ W * \mu - W * \nu \ \text{is a $\mathsf{P}$-martingale}}.
\end{equation*}
We call $\nu$ the compensator of $\mu$ (under $\mathsf{P}$).

\subsubsection{Characteristics of a semimartingale}
We fix $h(y) \coloneq y \mathbf{1}_{\abs{y} \leq 1}$.
The \emph{big-jumps removed} version of $X$ is
\begin{equation*}
  X^{h} \coloneq X - (y-h) * \mu.
\end{equation*}
If $X$ is a semimartingale under $\mathsf{P}$, then $X^{h}$ is a special semimartingale.
Thus, there exists a unique decomposition
\begin{equation*}
  X^{h} = X_{0} + M^{h} + B^{h},
\end{equation*}
with $M^{h}$ a martingale, and $B^{h}$ predictable.
We call $(B^{h}, C, \nu)$ the \emph{characteristics} of the semimartingale, where $C$ is the quadratic variation of the continuous part of $X$, and $\nu$ the compensator.
\begin{remark}
In general, it is not known whether a semimartingale with prescribed characteristics exists, nor if the characteristics characterise $\mathsf{P}$.
\end{remark}

\subsubsection{Representation of semimartingales with given characteristics}
Let $\mathsf{R}$ be a semimartingale with characteristics $(B^{h}, C, \nu)$.
Then, by \cite[Thm.\ II.2.34]{JacShi}, we have
\begin{equation*}
  X = X_{0} + X^{c} + h \ast (\mu^{X} - \nu) + (y - h) \ast \mu^{X} + B^{h}.
\end{equation*}

\subsection{Markov processes and related concepts}

\subsubsection{Shift semi-group}
We consider the \emph{shift semi-group}
\begin{equation*}
  Y_{s} \circ \theta_{t} = Y_{t+s}, \qquad t,s \in [0,1].
\end{equation*}
\subsubsection{Time reversal}
We consider the \emph{canonical time-reversed} process
\begin{equation*}
  X^{*}_{t} \coloneq X_{1-t}, \qquad t \in [0,1].
\end{equation*}
We write $\mathfrak{F}^{*}$ for the associated filtration.
\subsubsection{Markov process}
A process $\mathsf{P} \in \mathscr{P}(\Omega)$ is \emph{Markov} provided,
\begin{equation*}
  \mathsf{P}(X_{t+s} \in B \given \mathfrak{F}_{t}) = \mathsf{P}(X_{s} \in B \given X_{t}).
\end{equation*}
We shall use the following less common alternative characterization.
The process $\mathsf{P} \in $ is Markov if and only if
\begin{equation*}
  \mathsf{P}(C \cap C^{*} \given X_{t}) = \mathsf{P}(C \given X_{t}) \mathsf{P}(C^{*} \given X_{t}), \qquad t \in [0,1],\, C \in \mathfrak{F}_{t},\, C^{*} \in \mathfrak{F}^{*}_{t}.
\end{equation*}

\subsubsection{Additive functionals}
An adapted process $A$ is an \emph{additive functional} provided $A_{0} = 0$ and
\begin{equation*}
  A_{t+s} = A_{s} + A_{t} \circ \theta_{s}.
\end{equation*}
We identify every additive functional with the random additive set function by letting for $s < t$:
\begin{align*}
  & A_{[s,t]} \coloneq A_{t-s} \circ \theta_{s},
\\& A_{[s,t)} \coloneq A_{(t-s)^{-}} \circ \theta_{s},
\\& A_{(s,t]} \coloneq A_{(t-s)} \circ \theta_{s^{-}},
\\& A_{(s,t)} \coloneq A_{(t-s)^{-}} \circ \theta_{s^{-}}.
\end{align*}

\subsubsection{Reciprocal measure}
We say that $\mathsf{P}$ is \emph{reciprocal} provided
\begin{equation*}
  \mathsf{P}(X_{s} \in B \given \mathfrak{F}_{r}, \mathfrak{F}^{*}_{t}) = \mathsf{P}(X_{s} \in B \given X_{r}, X_{t}), \qquad r \leq s \leq t.
\end{equation*}
As for Markov processes, we repeatedly use the alternative characterization: $\mathsf{P}$ is reciprocal if and only if
\begin{equation*}
  \begin{split}
    & 0 \leq s \leq t \leq 1,\, C \in \mathfrak{F}_{s},\, C^{*} \in \mathfrak{F}^{*}_{t},\, A \in \sigma(X_{u} : s \leq u \leq t) \Rightarrow
  \\& \mathsf{P}(C \cap A \cap C^{*} \given X_{s}, X_{t}) = \mathsf{P}(C \cap A \given X_{s}, X_{t}) \mathsf{P}(C^{*} \cap A \given X_{s}, X_{t}).
  \end{split}
\end{equation*}
Clearly, all Markov processes are also reciprocal.
We have that whenever $\mathsf{P}$ is reciprocal, then $\mathsf{P}^{x}$ is Markov.

\subsection{Feller processes and their generators}
We follow \cite{BSW}.

\subsubsection{Feller semi-groups}
A \emph{Feller semi-group} is a Markov semi-group $(\mathbf{T}_{t})$ such that
\begin{equation*}
\norm{\mathbf{T}_{t} u - u}_{\infty} \xrightarrow[t \to 0]{} 0, \qquad u \in \mathscr{C}_{0}(\mathbb{R}^{d}).
\end{equation*}

\subsubsection{Infinitesimal generator}
The \emph{infinitesimal generator} of a Feller semi-group $(\mathbf{T}_{t})$ is the unique unbounded operator $\mathbf{A}$ with
\begin{align*}
  & \mathscr{D}(\mathbf{A}) \coloneq \set*{ u \in \mathscr{C}_{0}(\mathbb{R}^{d}) : \lim_{t \to 0} \frac{\mathbf{T}_{t}u -u}{t} \ \text{exists} },
\\& \mathbf{A} u \coloneq \lim_{t \to 0} \frac{\mathbf{T}_{t}u -u}{t}.
\end{align*}
The generator $\mathbf{A}$ is densely-defined and closed.

\subsubsection{Feller processes and martingales}
Let $\mathsf{P} \in \mathscr{P}(\Omega)$ be the law of a Feller process associated with the Feller semi-group $(\mathbf{T}_{t})$, that is
\begin{equation*}
  \mathbf{T}_{t}u(x) = \esp_{\mathsf{P}} \bracket*{ u(X_{t}) \given X_{0} = x }, \qquad x \in \mathbb{R}^{d},\, t \in [0,1].
\end{equation*}
Then, for all $f \in \mathscr{D}(\mathbf{A})$, the process
\begin{equation*}
  M_{t} \coloneq f(X_{t}) - f(X_{0}) - \int_{0}^{t} \mathbf{A}f(X_{s}) \mathtt{d} s, \qquad t \in [0,1],
\end{equation*}
is a $\mathsf{P}$-martingale.

\subsubsection{Lévy triplet of a Feller process}
In the rest of the paper, we always assume that $\mathscr{D}(\mathbf{A})$ contains the smooth functions $\mathscr{C}^{\infty}_{c}(\mathbb{R}^{d})$. 
In this case, by \cite{Courrege}, there exist:
\begin{enumerate}[(i)]
  \item $\alpha \colon \mathbb{R}^{d} \to \mathbb{R}_{+}$;
  \item $b^{h} \colon \mathbb{R}^{d} \to \mathbb{R}^{d}$;
  \item $c \colon \mathbb{R}^{d} \to \mathbb{R}^{d \times d}$ symmetric non-negative;
  \item $N \colon \mathbb{R}^{d} \to \mathscr{M}(\mathbb{R}^{d} \setminus \{0\})$ with $\int_{\mathbb{R}^{d}} (\abs{y}^{2} \wedge 1) N(x, \mathtt{d} y) < \infty$;
\end{enumerate}
such that
\begin{equation}\begin{split}
  \label{eq:generator-feller}
  \mathbf{A}u(x) &= -\alpha(x) u(x) + b^{h}(x) \cdot \nabla u(x) + \frac{1}{2} \nabla^2u(x) \cdot c(x) 
  \\&+ \int_{\mathbb{R}^{d} \setminus \{0\}} (u(x+y) - u(x) - \nabla u(x) \cdot h(y)) N(x, \mathtt{d} y).
\end{split}
\end{equation}

\subsubsection{Feller processes and semimartingales}
Let $\mathsf{P} \in \mathscr{P}(\Omega)$ be the law of a Feller process with Lévy triplet $(b^{h}, c, N)$.
Then $\mathsf{P}$ is a semimartingale with characteristics
\begin{equation}\label{eq:characteristics-feller}
  \begin{split}
  & B^{h}_{t} = \int_{0}^{t} b^{h}(X_{s}) \mathtt{d} s;
\\& C_{t} = \int_{0}^{t} c(X_{s}) \mathtt{d} s;
\\& \nu(\mathtt{d} s, \mathtt{d}y) = N(X_{s}, \mathtt{d}y) \mathtt{d} s.
  \end{split}
\end{equation}

\subsection{Examples of Feller semi-groups}

\subsubsection{Heat semi-group}
Let $\mathsf{R}$ be the law of the Wiener process on $\mathbb{R}^{d}$.
Consider the associated Feller semi-group:
\begin{equation*}
  \mathbf{T}_{t}u(x) \coloneq \int_{\mathbb{R}^{d}} \mathrm{e}^{-\abs{y-x}^{2}/2t} \frac{\mathtt{d} y}{(2\pi t)^{d/2}}.
\end{equation*}
Then $\mathbf{A} = -\frac{1}{2} \Delta$.
The invariant measure is the Lebesgue measure.

\subsubsection{Poisson semi-group}
Let $\mathsf{R} \in \mathscr{P}(\Omega)$ be the law of the Poisson process on $\mathbb{R}^{d}$ with intensity $\lambda > 0$ and jumping towards $z \in \mathbb{R}^{d}$.
The associated Feller semi-group is
\begin{equation*}
  \mathbf{T}_{t}u(x) \coloneq \sum_{k \in \mathbb{N}} u(x + kz) \frac{(\lambda t)^{j}}{j!} \mathrm{e}^{-t\lambda}.
\end{equation*}
Its generator is $\mathbf{A}u(x) = \lambda(u(x+z) - u(x))$.
The invariant measure is the counting measure.

\subsubsection{Symmetric stable semi-groups}
Let $\mathsf{R}$ be the law of the $\alpha$-stable symmetric, with $\alpha \in (0,2)$.
The associated Feller semi-group is 
\begin{equation*}
  \mathbf{T}_{t}u(x) \coloneq \int u(x+y) p_{\alpha,t}(\mathtt{d} y),
\end{equation*}
where $p_{\alpha,t}$ satisfies
\begin{equation*}
  \widehat{p}_{\alpha,t}(\xi) = \mathrm{e}^{-t \abs{\xi}^{\alpha}}.
\end{equation*}
Then the generator is
\begin{equation*}
  \mathbf{A}u(x) \coloneq k_{\alpha} \int_{\mathbb{R}^{d} \setminus \{0\}} (u(x+y) - u(x) - \nabla u(x) \cdot h(y)) \frac{\mathtt{d} y}{\abs{y}^{\alpha + d}}.
\end{equation*}
 
\subsubsection{Lévy processes}
Let $\mathsf{R}$ be the law of a Lévy process on $\mathbb{R}^{d}$, that is
\begin{equation*}
  \mathbf{T}_{t}u \coloneq u * p_{t},
\end{equation*}
where $(p_{t})$ is a family of infinitely divisible distributions.
In this case the Lévy triplet is independent of $x$, and the generator is
\begin{equation*}
  \mathbf{A}u(x) \coloneq b^{h} \cdot \nabla u(x) + \frac{1}{2} \nabla \cdot (c \nabla u(x)) + \int_{\mathbb{R}^{d} \setminus \{0\}} (u(x+y) - u(x) - \nabla u(x) \cdot h(y)) N(\mathtt{d} y).
\end{equation*}

\subsubsection{Ornstein--Uhlenbeck stable semi-groups}
Let $\mathsf{R}$ be the law of the $\alpha$-stable symmetric process.
The associated Ornstein--Uhlenbeck process $Y$ satisfies, under $\mathsf{R}$
\begin{equation*}
  \mathtt{d} Y_{t} = \mathtt{d} X_{t} - Y_{t} \mathtt{d} t.
\end{equation*}
By solving explicitly this equation, the associated semi-group is
\begin{equation*}
  \mathbf{T}_{t}u(x) \coloneq \esp_{\mathsf{R}} \bracket*{ u\paren*{\mathrm{e}^{t} x + \int_{0}^{t} \mathrm{e}^{s-t} \mathtt{d} X_{s}} \given X_{0} = x}.
\end{equation*}
In this case
\begin{equation*}
  \mathbf{A}u(x) = -(-\Delta)^{\alpha/2}u(x) - x \cdot \nabla u(x).
\end{equation*}
The unique invariant measure is $\mu_{\alpha}$ such that
\begin{equation*}
  \widehat{\mu}_{\alpha}(\xi) = \mathrm{e}^{-\frac{1}{\alpha} \abs{\xi}^{\alpha}}.
\end{equation*}

\subsection{Relative entropy and Girsanov formula}

\subsubsection{Relative entropy}
We fix $\mathsf{R} \in \mathscr{P}(\Omega)$ a semimartingale with characteristics $(B^{h}, C, \nu)$.
The \emph{relative entropy} of with respect to $R$ is the functional
\begin{equation*}
  \mathcal{H}\paren{\mathsf{P} \given \mathsf{R}} \coloneq 
  \begin{cases}
    \esp_{\mathsf{P}} \bracket*{\frac{\mathtt{d} \mathsf{P}}{\mathtt{d} \mathsf{R}}}, & \text{if}\ \mathsf{P} \ll \mathsf{R};
    \\ +\infty, & \text{otherwise}.
  \end{cases}
\end{equation*}
Elements of $\mathscr{P}(\Omega)$ that are absolutely continuous with respect to to a semimartingale are again a semimartingales \cite{JacShi}.
In particular, elements with finite entropy with respect to $\mathsf{R}$ are semimartingales.
\cite{Leonard} characterises their characteristics.

\subsubsection{Girsanov theorem under finite entropy}
We consider the functions $\theta \colon \mathbb{R} \ni u \mapsto \mathrm{e}^u-u-1$, and its convex conjugate
\begin{equation*}
  \theta^{\star} \colon \mathbb{R} \ni v \mapsto
  \begin{cases}
(v+1)\log(v+1) -v, & \text{if} \ v>-1;
\\ 1, & \text{if} \ v=-1;
\\ +\infty, & \text{otherwise}.
\end{cases}
\end{equation*}

\begin{theorem}[{\cite[Thms.\ 1 \& 3]{Leonard}}]\label{th:girsanov}
  Then, for all $\mathsf{P} \in \mathscr{P}(\Omega)$ such that $\mathcal{H}(\mathsf{P} \given \mathsf{R}) < \infty$, there exist an adapted process $\beta$, and predictable non-negative generalized process $\ell$ such that
  \begin{equation*}
    \esp_{\mathsf{P}} \int_{0}^{1} \beta_{s} \cdot C(\mathtt{d} s) \beta_{s} + \esp_{\mathsf{P}} \int_{0}^{1} \int_{\mathbb{R}^{d} \setminus \{0\}} \theta^{\star}(\abs{\ell_{s}(y)-1}) \nu(\mathtt{d}s \mathtt{d} y) < \infty.
  \end{equation*}
  Moreover, $\mathsf{P}$ is a semimartingale with characteristics $(B^{h} + \tilde{B}, C, \ell \nu)$,
  where
  \begin{equation*}
    \tilde{B}_{t} \coloneq \int_{0}^{t} C(\mathtt{d} s) \beta_{s} + \int_{0}^{t} \int_{\mathbb{R}^{d} \setminus \{0\}} h(y)(\ell_{s}(y) - 1) \nu(\mathtt{d} s \mathtt{d} y).
  \end{equation*}
\end{theorem}

\begin{remark}
  \cite{Leonard} states two different theorems, one continuous martingales, and another for pure jump martingales.
  Decomposing the semimartingale $\mathsf{P}$ in its continuous part and pure jump part, it is clear that the results carry out to the mixed case.
\end{remark}

\subsubsection{Density under finite entropy}
Additionally to the Girsanov theorem, \cite{Leonard} obtains an expression of the density in terms of the processes $\beta$ and $\ell$.
Consider the \emph{log-density process}
\begin{equation*}
  Z_t \coloneq \log \esp_{\mathsf{R}}\bracket*{\frac{\mathtt{d}\mathsf{P}}{\mathtt{d}\mathsf{R}}  \given \mathfrak{F}_t}.
\end{equation*}

\begin{proposition}[{\cite[Thms.\ 2 \& 4]{Leonard}}]\label{th:density-girsanov}
  On the event $\set*{\frac{\mathtt{d}\mathsf{P}}{\mathtt{d}\mathsf{R}}>0 }$, we have $Z = Z^{c} + Z^{+} + Z^{-}$, where, under $\mathsf{P}$
\begin{align*}
  & Z^{c}_{t} \coloneq \int_{0}^{t} \beta_{s} \cdot (\mathtt{d} X_{s} - \mathtt{d} B^{h}_{s} - C(\mathtt{d}s) \beta_{s}) + \frac{1}{2} \int_{0}^{t} \beta_{s} \cdot C(\mathtt{d} s) \beta_{s};
\\& Z^{+}_{t} \coloneq  \int_{[0,t]\times\mathbb{R}^d_*} \mathbf{1}_{\set*{\ell\geq\frac{1}{2}}} \, \log \ell \, \mathtt{d}(\mu^X-\ell\nu) + \int_{[0,t] \times \mathbb{R}^d_*} \mathbf{1}_{\set*{\ell\geq\frac{1}{2}}} \, \theta(\ell -1) \, \mathtt{d}\nu
\\& Z^{-}_{t} =  \int_{[0,t] \times \mathbb{R}^d_*} \mathbf{1}_{\set*{0\leq\ell<\frac{1}{2}}} \, \log \ell \, \mathtt{d} \mu^X + \int_{[0,t]\times\mathbb{R}^d_*} \mathbf{1}_{\set*{0\leq\ell<\frac{1}{2}}}\, (-\ell+1)\, \mathtt{d}\nu.
\end{align*}
\end{proposition}

\begin{remark}
  \begin{enumerate}[(a),wide,nosep]
    \item As noticed in \cite{Leonard}, on the event $\left\{\frac{dP}{dR}>0\right\}$, $\ell>0$ and the sum in the first integral of the definition of $Z^-$ is finite.
    \item The exact value $1/2$ for the cut-off between $Z^{+}$ and $Z^{-}$ is artificial and could be chosen anywhere in $(0,1)$.
    \item This expression of $Z$ does not provide a decomposition as the sum of $P$-local martingale and absolutely continuous process.
      Indeed, the stochastic integral $(\log\ell)*(\mu^X-\ell\nu)$ is meaningless, in general.
      Additional assumptions could guarantee that it makes sense.
  \end{enumerate}
\end{remark}

\subsubsection{Disintegration with respect to the initial condition}
So far $\mathsf{P}_{0}$ could be arbitrary, we consider the regular disintegration of $\mathsf{P}$ along its marginal $\mathsf{P}_{0}$.
In this way, there exists a family $(\mathsf{P}^{x})_{x \in \mathbb{R}^{d}}$ such that $\mathsf{P}^{x}$ is supported on $\{ X_{0} = x\}$ and
\begin{equation*}
  \mathsf{P} = \int \mathsf{P}^{x} \mathsf{P}_{0}(\mathtt{d} x).
\end{equation*}
The semimartingale property, the Markov property, the special semimartingale property, and the Feller property are stable under this conditioning.
\begin{remark}
In particular, if $\mathsf{R}$ is a Markov measure, then all the $\mathsf{R}^{x}$ are also Markov.
The data $(\Omega, \mathfrak{F}, (X_{t})_{t \in [0,1]}, (\mathsf{R}^{x})_{x \in \mathbb{R}^{d}})$ is sometimes what is called a Markov process.
\end{remark}

\subsubsection{Chain rule for the entropy}
By the chain rule for the entropy, \cite[Thm.\ D.13]{DemboZeitouni}, it turns out that
\begin{equation}\label{eq:chain-rule-entropy}
  \mathcal{H}(\mathsf{P} \given \mathsf{R}) = \mathcal{H}(\mathsf{P}_{0} \given \mathsf{R}_{0}) + \int \mathcal{H}(\mathsf{P}^{x} \given \mathsf{R}^{x}) \mathsf{P}_{0}(\mathtt{d} x).
\end{equation}
In particular, if $\mathcal{H}(\mathsf{P} \given \mathsf{R}) < \infty$, and that $\mathsf{R}$ is a semimartingale, then all the $\mathsf{P}^{x}$ ($x \in \mathbb{R}^{d}$) are also semimartingales.

\section{The Brenier--Schrödinger problem with respect to a Feller semimartingale}

\subsection{Formulation of the problem}
\subsubsection{Setting}
In the rest of the paper, we fix $\mathsf{R} \in \mathscr{P}(\Omega)$ a Feller semimartingale with characteristics as in \cref{eq:characteristics-feller}.
Let $\pi\in\mathscr{P}(\mathbb{R}^n\times \mathbb{R}^n)$, that we interpret as a coupling between the initial and the final position.
Let $(\mu_t)_{t \in [0,1]} \in \mathscr{P}(\mathbb{R}^n)^{[0,1]}$, that we interpret as the incompressibility condition.
We study the Brenier--Schrödinger minimisation problem, with respect to the measure $R$  
\begin{equation}\label{eq:BS}\tag{LBS}
  \Inf*{ \mathcal{H}(\mathsf{P} \given \mathsf{R}) : \mathsf{P} \in \mathscr{P}(\Omega),\ \forall t \in [0,1],\, P_t = \mu_t,\ P_{01}=\pi }
\end{equation} 
\subsubsection{General shape of the solutions}
By \cite[Thm.\ 4.7]{ACLZ}, every minimizer $\mathsf{P}$ in \cref{eq:BS} is a reciprocal measure of the form
\begin{equation*}
  \mathsf{P} = \Exp*{ \eta(X_0,X_1) + A_{1} } \mathsf{R},
\end{equation*}
for some Borel function $\eta \colon \mathbb{R}^{d} \times \mathbb{R}^{d} \to \mathbb{R}$ and an additive functional $A$.
Conditioning at $\{X_{0} = x\}$, we find that
\begin{equation*}
  \mathsf{P}^{x} = \Exp*{ \eta(x,X_{1}) + A_{1} } \mathsf{R}^{x}.
\end{equation*}
Using the additive property of $A$, we find that
\begin{equation*}
  A_{1} = A_{[0,1]} = A_{t^{-}} + A_{[t,1]}.
\end{equation*}
It follows that the log-density process looks like
\begin{equation*}
  Z^{x}_{t} \coloneq \log \esp_{\mathsf{R}^{x}} \bracket*{ \frac{\mathtt{d}\mathsf{P}^{x}}{\mathtt{d}\mathsf{R}^{x}} \given \mathfrak{F}_{t} } = A_{t^{-}} + \log \esp_{\mathsf{R}^{x}} \bracket*{ \exp(x, X_{1}) + A_{[t,1]} \given X_{t} }.
\end{equation*}
In view of this formula, let us define
\begin{equation}\label{eq:potential}
  \psi^{x}(t,z) \coloneq \log \esp_{\mathsf{R}^x}\bracket*{ \Exp*{ \eta(x,X_1) + A_{[t,1]} } \given X_t = z}, \qquad t \in [0,1], z \in \mathbb{R}^{d}.
\end{equation}
With this definition, the above expression for $Z^{x}$ reads
\begin{equation*}
  Z_{t}^{x} = A_{t^{-}} + \psi^{x}(t, X_{t}).
\end{equation*}

\subsection{Non-local Hamilton--Jacobi--Bellman equation}

Let us reformulate the main result of the paper with our introduced terminology and give the prove.
\begin{theorem}\label{th:hjb}
  Let $\mathsf{P} \in \mathscr{P}(\Omega)$ be a regular solution to \cref{eq:BS}.
  Then, there exists a function $p \colon [0,1] \times \mathbb{R}^{d}$, such that the potential $\psi^{x}$ defined in \cref{eq:potential}, is a strong solution to
  \begin{equation*}
    \partial_{t} \psi^{x}_{t}(z)+ \mathrm{e}^{-\psi^{x}_{t}(z)} (\mathbf{A} \mathrm{e}^{\psi^{x}_{t}})(z) + \frac{1}{2} \nabla \psi_{s}^{x}(z) \cdot c(z) \nabla \psi_{s}^{x}(z) + p_{t}(z) = 0.
  \end{equation*}
\end{theorem}

The proof follows from two different representations of $Z^{x}$ as a special martingale.

\subsubsection{Decomposition using Itō's formula}

For short, write
\begin{equation*}
  J_{s}^{x}(y) \coloneq \psi_{s}^{x}(X_{s^{-}} + y) - \psi_{s}^{x}(X_{s^{-}}).
\end{equation*}
\begin{lemma}\label{th:decomposition-ito}
  Under $\mathsf{P}^{x}$, the process $\psi^{x}_{\cdot}(X)$ is a special martingale.
  Moreover, we have
  \begin{equation*}
    \begin{split}
      \psi_{t}^{x}(X_{t}) &= \psi_{t}^{x}(x) + \int_{0}^{t} \nabla \psi_{s}^{x}(X_{s^{-}}) \cdot \mathtt{d} X^{c}_{s} + J * (\mu - \ell^{x} \nu) + \int_{0}^{t} \partial_{s} \psi_{s}^{x}(X_{s^{-}}) \mathtt{d} s 
                        \\&+ \int \nabla \psi_{s}^{x}(X_{s^{-}}) \cdot \mathtt{d} \bar{B}_{s} + \frac{1}{2} \int_{0}^{t} \nabla^{2} \psi_{s}^{x}(X^{s^{-}}) \cdot \mathtt{d} C_{s} + \int_{0}^{t} [J_{s}(y) - \nabla \psi_{s}^{x}(X_{s^{-}}) \cdot h(y)] \nu(\mathtt{d}s \mathtt{d}y).
    \end{split} 
    \end{equation*}
  \end{lemma}
\begin{proof}
  Since by assumption, $\mathsf{R}^{x}$ is a Feller semimartingale with characteristics $(B^{h}, C, \nu)$, by \cref{eq:chain-rule-entropy,th:girsanov}, we find that $\mathsf{P}^{x}$ is a semimartingale.
  Moreover, there exist an adapted process $\beta^{x}$ and a predictable non-negative process $\ell^{x}$ with characteristics $(B^{h} + \tilde{B}^{x}, C, \ell^{x} \nu)$.
  For short, let us consider the local martingale $M$, the drift $\bar{B}$, and the generalized processes $J$ and $W$ defined by
  \begin{align*}
    & M \coloneq X^{c} + h \ast (\mu - \ell^{x} \nu);
  \\& \bar{B} \coloneq B^{h} + B^{x};
  \\& J_{s}(y) \coloneq \psi_{s}^{x}(X_{s^{-}} + y) - \psi_{s}^{x}(X_{s^{-}});
  \\& W_{s}(y) \coloneq J_{s}(y) - \nabla \psi_{s}^{x}(X_{s^{-}}) \cdot h(y).
  \end{align*}
  Since, by assumption, $\psi^{x}(X)$ is a special semi-martingale, the same argument as in \cite[Thm.\ II.2.42]{JacShi} yields that $W * \mu - W * (\ell^{x} \nu) = W * (\mu - \ell^{x} \nu)$ is a local martingale, and
  \begin{equation*}
    \begin{split}
      \psi_{t}^{x}(X_{t}) &= \psi_{t}^{x}(x) + \int_{0}^{t} \partial_{s} \psi_{s}^{x}(X_{s^{-}}) \mathtt{d} s + \int_{0}^{t} \nabla \psi_{s}^{x}(X_{s^{-}}) \cdot \mathtt{d} M_{s} + \int \nabla \psi_{s}^{x}(X_{s^{-}}) \cdot \mathtt{d} \bar{B}_{s} 
                        \\&+ \frac{1}{2} \int_{0}^{t} \nabla^{2} \psi_{s}^{x}(X_{s^{-}}) \cdot \mathtt{d} C_{s} +  W \ast (\mu - \ell^{x}\nu) + W * (\ell^{x} \nu).
    \end{split}
  \end{equation*}
  We conclude by observing that
  \begin{equation*}
    \int_{0}^{t} \nabla \psi_{s}^{x}(X_{s^{-}}) \cdot \mathtt{d} (h * (\mu - \ell^{x} \nu)) + W * (\mu - \ell^{x} \nu) = J * (\mu - \ell^{x} \nu).
  \end{equation*}
\end{proof}

\subsubsection{Decomposition using Girsanov's theorem}

\begin{lemma}\label{th:decomposition-girsanov}
  The process $(\log \ell^{x}) * (\mu - \ell^{x} \nu)$ is a well-defined local martingale.
  Moreover, under $\mathsf{P}^{x}$,
\begin{equation*}
  \begin{split}
    Z^{x}_t &= (\log \ell^{x}) * (\mu - \ell^{x} \nu) + \int_{0}^{t} \beta^{x}_{s} \cdot \mathtt{d} X^{c}_{s} 
          \\&+ \int_{0}^{t} \int \theta^{\star}(\ell^{x}_{s}(y) - 1) N(X_{s}, \mathtt{d} y) \mathtt{d} s + \frac{1}{2} \int_{0}^{t} \beta^{x}_{s} \cdot c(X_{s}) \beta^{x}_{s} \mathtt{d} s,
  \end{split}
\end{equation*}
where the first line on the right-hand side is a local martingale, and the second line if a predictable process.
\end{lemma}

\begin{proof}
  The regular solution assumption ensures that the measure $\mathsf{P}^x$ and $\mathsf{R}^x$ are equivalent. 
  The part involving $\beta^{x}$ follows directly from \cref{th:density-girsanov}. 
  In view of \cref{th:density-girsanov}, it is sufficient to show that
  \begin{equation*}
    Z_{t}^{+} + Z_{t}^{-} = (\log \ell^{x}) * (\mu - \ell^{x}\nu) + \theta^{\star}(\ell^{x} - 1) * \nu.
  \end{equation*}
  Recall that
  \begin{equation}\label{eq:decomposition-density}
  \begin{split}
    Z_t^{+} + Z_{t}^{-} = &\int_{[0,t]\times\mathbb{R}^d_*}\mathbf{1}_{\ell^{x} \geq\frac{1}{2}}\log(\ell^{x})\, \mathtt{d}(\mu-\ell^{x}\nu) + \int_{[0,t]\times\mathbb{R}^d_*}\mathbf{1}_{\ell^{x} \geq \frac{1}{2}} \theta^{\star}(\ell^{x} - 1) \, \mathtt{d}\nu
                        \\& + \int_{[0,t]\times\mathbb{R}^d_*}\mathbf{1}_{0\leq \ell < \frac{1}{2}} \log(\ell^{x})\, \mathtt{d} \mu + \int_{[0,t]\times\mathbb{R}^d_*}\mathbf{1}_{0 \leq \ell^{x} < \frac{1}{2}}(-\ell^{x} + 1)\, \mathtt{d}\nu.
  \end{split}
\end{equation}
Since $\ell^{x} \mathbf{1}_{\ell^{x} \geq2} \log(\ell^{x})$ and $\ell^{x} \, (\mathbf{1}_{1/2\leq \ell<2}\log(\ell^{x}))^{2}$ are dominated by $\theta^*(|\ell-1|)$, which is $\mathsf{P}{x}$-integrable by \cref{th:girsanov}, thus, by \cite[II.1.27, p.\ 72]{JacShi}, $\mathbf{1}_{\ell^{x} \geq 1/2} (\log \ell^{x})$ is a local martingale.
The second and the fourth terms have $\mathsf{P}^{x}$-integrable variations because they are also dominated by $\theta^*(|\ell-1|)*\nu$. 
Since $Z^{x}$ is a special semimartingale, by \cite[I.2.24]{JacShi}, the third term also have $\mathsf{P}^{x}$-integrable variations.
This implies that the stochastic integral $\mathbf{1}_{0\leq \ell <1/2}\log(\ell)*(\mu^X-\ell\nu)$ is well-defined and a local martingale.
Moreover by \cite[II.1.28]{JacShi}, we have
\begin{equation*}
  \mathbf{1}_{0 \leq \ell^{x} <1/2}\log(\ell^{x}) * (\mu - \ell^{x} \nu) = \mathbf{1}_{0\leq \ell^{x} < 1/2}\log(\ell^{x}) * \mu - \mathbf{1}_{0\leq \ell^{x} < 1/2}\log(\ell^{x}) * (\ell^{x} \nu).
\end{equation*}
The proof is thus complete by compensating the third term in \cref{eq:decomposition-density}, and by using that the map $a \mapsto a * (\mu - \ell^{x}\nu)$ is linear.
\end{proof}

\subsubsection{Conclusion}

\begin{proof}[Proof of {\cref{th:hjb}}]
  By the uniqueness of the decomposition of a special semimartingale, comparing the expression for $Z^{x}$ in \cref{th:decomposition-ito,th:decomposition-girsanov}, we obtain by identification of the local martingale parts
  \begin{equation}\label{eq:indentification-martingale}
   \log \ell^{x}_{s}(y) = \psi^{x}_{s}(X_{s^{-}} + y) - \psi^{x}_{s}(X_{s^{-}}), \qquad \text{and} \qquad
  \beta_{s}^{x} = \nabla \psi^{x}_{\cdot}(X_{s^{-}}).
  \end{equation}
By identification of the predictable part, we obtain
\begin{equation*} 
  \begin{split}
    \theta^{\star}(\ell^{x} - 1) * \nu_{s} &+ \frac{1}{2} \int_{0}^{t} \beta^{x}_{s} \cdot C(\mathtt{d} s) \beta^{x}_{s} = A_{s^{-}} + \psi^{x}_{0}(x) + \int_{0}^{t} \partial_{s} \psi^{x}_{s}(X_{s^{-}})\, \mathtt{d} s 
                                                                                                                      \\&  + \frac{1}{2} \int_{0}^{t} \nabla^{2} \psi_{s}^{x}(X^{s^{-}}) \cdot \mathtt{d} C_{s} + \int_{0}^{t} \nabla \psi^{x}_{s}(X_{s^{-}}) \cdot \mathtt{d} B^{h}_{s} + \int_{0}^{t} \nabla \psi_{s}^{x}(X_{s^{-}}) C(\mathtt{d} s) \beta^{x}_{s} 
                                                                                                                      \\&+ \int _{[0,t] '*\mathbb{R}^{d} \setminus \{0\}} \nabla \psi_{s}^{x}(X_{s^{-}}) \cdot h  \mathtt{d} (\ell^{x}-1)\nu + \int_{[0,t]\times\mathbb{R}^d_*} [J_s -\nabla \psi^{x}_{s}(X_{s^{-}})\cdot h]\, \mathtt{d}(\ell^{x}\nu).
  \end{split}
\end{equation*}
  Reporting \cref{eq:indentification-martingale} in the above, terms simplify and we arrive at
  \begin{equation*}
    \begin{split}
    &-\int_{[0,t]\times\mathbb{R}^d_*} \left[\mathrm{e}^{J_s}-1-\nabla \psi^x_s(X_s) \cdot h(y)\right]\, \mathtt{d} \nu - \frac{1}{2} \int_{0}^{t} \nabla^{2} \psi_{s}^{x}(X_{s^{-}}) \cdot \mathtt{d} C_{s} - \frac{1}{2} \int_{0}^{t} \beta_{s}^{x} \cdot C(\mathtt{d} s) \beta_{s}^{x} 
    \\& = A_{t^{-}} + \psi^{x}_{0}(x) + \int_{0}^{t} \partial_{s} \psi^{x}_{s}(X_{s^{-}})\, \mathtt{d}s
    + \int_{0}^{t} \nabla \psi^{x}_{s}(X_{s^{-}}) \cdot \mathtt{d} B^{h}_s\\
    \end{split}
  \end{equation*} 
  Recalling that the characteristics have the particular given in \cref{eq:characteristics-feller}, we see, on the one hand, that $A$ is actually absolutely continuous.
  Since $A$ is an absolutely continuous additive functional there exist, there exists a pressure $p \coloneq [0,1]\times\mathbb{R}^d$ such that 
  \begin{equation*}
  A_{t} = \int_{0}^{t} p_{s}(X_{s}) \, \mathtt{d}s.
\end{equation*}

  On the other hand, by differentiating with respect to $t$, we obtain
  \begin{equation*}
    \begin{split}
    &  \partial_{s} \psi^{x}_{s}(X_{s^{-}}) + p_{s}(X_{s}) = - \frac{1}{2} \nabla^{2} \psi_{s}^{x}(X_{s^{-}}) \cdot c(X_{s}) - \frac{1}{2} \nabla \psi_{s}^{x}(X_{s^{-}}) \cdot c(X_{s}) \nabla \psi_{s}^{x}(X_{s^{-}})
  \\& -\mathrm{e}^{-\psi^x_t(X_t)}\int_{\mathbb{R}^d_*}\left[\mathrm{e}^{\psi^x_t(X_t+y)} - \mathrm{e}^{\psi^x_t(X_t)} - \nabla \mathrm{e}^{\psi^x_t(X_t)} \cdot h(y)\right]\, N(X_{s}, \mathtt{d}y) + \nabla \psi^{x}_{s}(X_{s^{-}}) \cdot  b^{h}(X_{s^{-}}). 
    \end{split} 
\end{equation*}
This last equation is true $\mathsf{P}^x$ almost surely.
To conclude from there, we use a continuity argument and the equivalence of $\mathsf{P}^x$ and $\mathsf{R}^x$.
\end{proof}

\section{Existence of solutions for the Ornstein-Uhlenbeck problem}

Since the functional $\mathcal{H}(\cdot \given \mathsf{R})$ is convex and lower semi-continuous on $\mathscr{P}(\Omega)$, by the direct method of the calculus of variations, we obtain immediately the following result.
\begin{lemma}
  Assume that there exist $\mathsf{P} \in \mathscr{P}(\Omega)$, such that $\mathsf{P}_{01} = \pi$, $\mathsf{P}_{t} = \mu_{t}$ for all $t \in [0,1]$, and $\mathcal{H}(\mathsf{P} \given \mathsf{R}) < \infty$.
  Then, there exists a unique minimizer to \cref{eq:BS}.
\end{lemma}
Thus, in relation with \cref{th:hjb}, we have to answer two questions:
\begin{enumerate}[(A)]
  \item\label{question-existence} Can we find a candidate $\mathsf{P} \in \mathscr{P}(\Omega)$ for \cref{eq:BS} with finite entropy?
  \item\label{question-regularity} Is the unique solution to \cref{eq:BS} regular?
\end{enumerate}

\cite[Prop.\ 6.1]{ACLZ} gives a positive answer to \ref{question-existence}, when $\mathsf{R}$ is the reversible Brownian motion on the torus, $\mu_{t} = \mathsf{vol}$ for all $t \in [0,1]$ (incompressible case), and $\pi$ is any coupling satisfying $\mathcal{H}(\pi \given \mathsf{R}_{01}) < \infty$.
\cite{GarciaZeladaHuguet} studies the case of the reflected Brownian measure on some quotient spaces, and of a non-reversible Brownian measure in $\mathbb{R}^{n}$.
In the latter case, the incompressible condition translates to Gaussian marginal constraints.
All the results motioned above, the necessary and sufficient condition for existence is $\mathcal{H}(\pi|\mathsf{vol}^{2}) < +\infty$.
In the Ornstein-Uhlenbeck case, we only manage to prove existence when $\pi$ is a Gaussian of a certain form.

In this section, we obtain a positive answer to \cref{question-existence} when $\mathsf{R}$ is a reversible Brownian Ornstein--Uhlenbeck process on $\mathbb{R}$.
In the language of semimartingales , $\mathsf{R}$ as characteristics $(-X_t,\sqrt{2}, 0)$, and is started from a Gaussian distribution.
In terms of stochastic differential equations, under $\mathsf{R}$, the canonical process $X$ satisfies 
\begin{equation*}
  \begin{cases}
  & \mathtt{d} X_t = \sqrt{2}\mathtt{d} W_t -X_t \mathtt{d} t,
\\& X_{0} \sim \gamma \coloneq \mathcal{N}(0,1),
  \end{cases}
\end{equation*}
where $W$ is an Brownian motion, under $\mathsf{R}$, independent of $X_0$.
The measure $\gamma$ is the unique invariant measure of this process, thus we study \cref{eq:BS} under the natural incompressible condition $\mu_{t} \coloneq \gamma$ for all $t \in [0,1]$.
Hence, we consider the minimisation problem
Namely, our minimisation problem in this specific case only depends on the parameter $\pi \in \mathscr{P}(\mathbb{R}^{n} \times \mathbb{R^{n}})$, and reads
\begin{equation}\label{eq:BS-OU}
  \Inf*{ \mathcal{H}(\mathsf{P} \given \mathsf{R}) : \mathsf{P} \in \mathscr{P}(\Omega), \mathsf{P}_{t} = \gamma \, \forall t \in [0,1], \mathsf{P}_{01} = \pi }.
\end{equation}
For $|c| \leq 1$, let us write
\begin{equation*}
  \gamma_{c} \coloneq \mathcal{N}\paren*{0,\begin{pmatrix}1&c\\c&1\end{pmatrix}}
\end{equation*}
Our results in this case is as follows.
\begin{proposition}\label{th:unique-solution-ou}
  Let $\pi \coloneq \gamma_{c}$ with $4e^{-1}-3e^{-1/3}\leq c<1$.
Then, the problem \eqref{eq:BS-OU} admits a unique solution.
\end{proposition}

\begin{remark}
  Since $\mathsf{P}_{0} = \mathsf{P}_{1} = \gamma$, we necessarily have that the variance under $\pi$ is $1$.
  Also if $|c| \coloneq 1$, then $\pi$ is degenerated and $\mathcal{H}(\pi|\mathsf{R}_{01})$ is not finite and so the problem would not have any solution.
  Lastly, for the particular $c \coloneq e^{-1}$, since $\pi=\mathsf{R}_{01}$, the problem admits the trivial solution $P=R$.
\end{remark}

As in \cite{ACLZ,GarciaZeladaHuguet}, we create candidate path measures as mixture of $\mathsf{R}$-bridges.
In the setting of the Ornstein-Uhlenbeck process, we exploit the following  explicit representation for the bridge
\begin{equation*}
  \mathsf{R}^{xy} \coloneq \mathsf{R}(\cdot \given X_{0} = x, X_{T} = y), \qquad x,y \in \mathbb{R},\, T \in [0,1].
\end{equation*}

\begin{lemma}[{\cite{BarczyKern}}]\label{th:bridge-representation-ou}
  The Ornstein--Uhlenbeck bridge $\mathsf{R}^{xy}$ coincides with the law of the process
\begin{equation*}
  U^{x,y}_t \coloneq \frac{\sinh(T-t)}{\sinh(T)}x +\frac{\sinh(t)}{\sinh(T)}y + \sqrt{2}\int_0^t\frac{\sinh(T-t)}{\sinh(T-s)}\, \mathtt{d} W_s,
\end{equation*}
where $W$ is a standard Brownian motion.
In particular,
\begin{equation*}
  \mathsf{R}^{xy}_{t} = \mathcal{N} \paren*{ \frac{\sinh(T-t)}{\sinh(T)}x +\frac{\sinh(t)}{\sinh(T)}y, 2\frac{\sinh(T-t)\sinh(t)}{\sinh(T)} }.
\end{equation*}
\end{lemma}

Let $T>0$ and $\sigma\in\mathscr{P}(\mathbb{R}^2)$.
We define
\begin{equation}\label{eq:bridge-mixture-2}
  \mathsf{Q} \coloneq \int_{\mathbb{R}^2} \mathsf{R}^{xy} \sigma(\mathtt{d}x\mathtt{d}y).
\end{equation}
The path measure $\mathsf{Q}$ is a mixture of Ornstein--Uhlenbeck bridges.
\begin{remark}
  For Brownian bridges, \cite{ACLZ,GarciaZeladaHuguet} can choose for $\sigma$ a product measure such that the mixture $\mathsf{Q}$ satisfies the incompressibility condition, that is $\mathsf{Q}_{t} = \mathsf{vol}$ or $\mathsf{Q}_{t} = \mathcal{N}(0,1/4)$, for all $t \in [0,1]$.
  For Ornstein-Uhlenbeck bridges, choosing $\sigma$ as a product \emph{cannot} yield an invariant process $\mathsf{Q}$.
  This explains why we need to introduce correlations, and why we are this limited to Gaussian couplings for $\pi$. 
\end{remark}

\begin{lemma}\label{th:invariant-bridge}
  Consider the bridge mixture $\mathsf{Q}$ as defined in \cref{eq:bridge-mixture-2}, with $\sigma \coloneq \gamma_{\rho}$ for some $\abs{\rho} < 1$.
  Then, $\mathsf{Q}_{t} = \gamma$ for all $0\leq t\leq T$ if and only if $\rho= \mathrm{e}^{-T}$.
\end{lemma}

\begin{proof}
  If $\rho=e^{-T}$, then $\sigma=\mathsf{R}_{0T}$ and $\mathsf{Q}=\mathsf{R}$.
  Let us show that it is the only possible $\rho$.
  Let $0\leq t\leq T$.
  According to \cref{th:bridge-representation-ou} and the definition of $\mathsf{Q}$, $\mathsf{Q}_t$ is the law of
\begin{equation*}
\frac{\sinh(T-t)}{\sinh(T)}X +\frac{\sinh(t)}{\sinh(T)}Z + \sqrt{\frac{\sinh(T-t)\sinh(t)}{\sinh(T)}}W,
\end{equation*}
where $(X,Z) \sim \gamma_{\rho}$ and $W$ is an independent standard Gaussian random variable.
In particular, we find for the variance
\begin{equation*}
\Var*{Q_t} = \frac{\sinh^2(T-t)}{\sinh^2(T)} +2\rho\frac{\sinh(T-t)\sinh(t)}{\sinh^2(T)}+\frac{\sinh^2(t)}{\sinh^2(T)} + 2\frac{\sinh(T-t)\sinh(t)}{\sinh(T)}.
\end{equation*}

So, the variance is constant and equals $1$ if and only if for all $0<t<T$, we have
\begin{equation*}
2\rho\sinh(T-t)\sinh(t)
=\sinh^2(T) -\sinh^2(T_t)-\sinh^2(t)-2\sinh(T)\sinh(T-t)\sinh(t).
\end{equation*}
By direct computations, the right-hand side becomes
\begin{equation*}
2\sinh(T-t)\sinh(t)(\cosh(T)-\sinh(T)).
\end{equation*}
Thus, the variance is constant and equals to $1$ if and only if $\rho = \cosh(T) - \sinh(T) = \mathrm{e}^{-T}$.
\end{proof}

\begin{proof}[Proof of {\cref{th:unique-solution-ou}}]
Actually, we need to concatenate several bridges in order to conclude.
In this way, we obtain a free parameter for us to optimise.
We let $r \coloneq e^{-1/3}$, and $s \in \mathbb{R}$ to be chosen later.
Let $\sigma\in\mathscr{P}(\mathbb{R}^4)$ be the centred Gaussian law with covariance
\begin{equation*}
  C \coloneq \begin{pmatrix} 1&r&s&c\\ r&1&r&s\\ s&r&1&r\\ c&s&r&1\\ \end{pmatrix},
\end{equation*}
and $\mathsf{Q} \in \mathscr{P}(\Omega)$ defined by
\begin{equation*}
  \mathsf{Q} \coloneq \int_{\mathbb{R}^3} \mathsf{R}(\cdot \given X_0=x, X_{1/3}=u, X_{2/3}=v, X_1=y)\sigma(\mathtt{d}x\mathtt{d}u\mathtt{d}v\mathtt{d}y).
\end{equation*}

\subsubsection*{The measure $Q$ has finite relative entropy.}
By the chain rule for the entropy \cref{eq:chain-rule-entropy}, we have 
\begin{equation*}
  \mathcal{H}(\mathsf{Q} \given \mathsf{R}) = \mathcal{H}(\pi \given \mathsf{R}_{01}) + \int \mathcal{H}(\sigma^{xy}|\mathsf{R}^{xy}_{1/3,2/3})\pi(\mathtt{d}x\mathtt{d}y).
\end{equation*}
Since, $\pi$ and $\mathsf{R}_{01}$ on the one hand, and $\sigma^{xy}$ and $\mathsf{R}^{xy}_{1/3,2/3}$ on the other hand, are non-degenerated Gaussian laws, their relative entropies are finite.
Furthermore, $\mathcal{H}(\sigma^{xy} \given \mathsf{R}^{xy}_{1/3,2/3})$ is a quadratic polynomial in $x$ and $y$.
Thus it is integrable with respect to the Gaussian measure $\pi$.
\subsubsection*{The measure $\mathsf{Q}$ satisfies the marginal conditions}
By construction we have that $\mathsf{Q}_{01}=\pi$.
Let $0<t<1$.
Since $\mathsf{R}$ is a reciprocal measure, whenever $h\in\{0, 1/3, 2/3\}$
\begin{equation*}
  \mathsf{Q}_t = \int \mathsf{R}_t(\cdot \given X_h=x, X_{h+1/3}=y) \gamma_{r}(\mathtt{d}x\mathtt{d}y),
\end{equation*}
Hence, using \cref{th:invariant-bridge}, we have $\mathsf{Q}_t = \gamma$.

\subsubsection*{Handling the parameters}
To conclude, let us derive conditions on $s$ and $c$, under which $C$ is a covariance matrix, that is positive definite.
Since $C$ is a Toeplitz matrix, its eigenvalues are
\begin{align*}
  & \frac{1}{2}(c+r+2\pm\sqrt{c^2-2cr+5r^2+8rs+4s^2}),
\\& \frac{1}{2}(-c-r+2\pm\sqrt{c^2-2cr+5r^2-8rs+4s^2}).
\end{align*}
Thus, $C$ is a covariance matrix if and only if
\begin{equation*}
\frac{s^2+2rs+r^2-r-1}{r+1}<c<\frac{s^2-2rs+r^2+r-1}{r-1}.
\end{equation*}
These two inequalities have solutions if and only if $s \in (2r^2-1, 1)$.
Then for each $r(4r^2-3)<c<1$, there exists $s\in (2r^2-1, 1)$ such that $\Gamma$ is the covariance matrix of a non-degenerated Gaussian measure.
This proves the existence of a unique solution. 
\end{proof}

\begin{remark}
Our candidate measure $Q$ is slightly more involved than the one from \cite{ACLZ}, where the bridge is only conditioned at the time $1/2$.
In our case, their approach would only prove the existence of solutions for $2e^{-1}-1\leq c<1$.
Conditioning at times $1/3$ and $2/3$ gives more flexibility, thanks to the free parameter $s$.
\end{remark}

\printbibliography%
\end{document}